\documentclass[a4paper, 2pt]{article}
\usepackage{amsmath}
\usepackage{amssymb}

\newtheorem{theorem}{Theorem}

\newtheorem{corollary}[theorem]{Corollary}

\newtheorem{proposition}[theorem]{Proposition}
\newtheorem{remark}[theorem]{Remark}

\newenvironment{proof}[1][Proof]{\noindent\textbf{#1.} }{\ \rule{0.5em}{0.5em}}
\input{tcilatex}
\begin{document}

\title{On the symmetry group of the $n$-dimensional Berwald-Mo\'{o}r metric}
\author{Hengameh Raeisi-Dehkordi\thanks{{\small Department of Mathematics
and Computer Sciences, Amirkabir University of Technology (Tehran
Polytechnic), 424 Hafez Ave., 15875-4413 Tehran, Iran, E-mail:
hengameh\_62@aut.ac.ir}}\ \ and Mircea Neagu\thanks{{\small Department of
Mathematics and Informatics, University Transilvania of Bra\c{s}ov, 50 Iuliu
Maniu Blvd., 500091 Bra\c{s}ov, Romania, E-mail: mircea.neagu@unitbv.ro%
\newline
}}}
\date{}
\maketitle

\begin{abstract}
In this paper we parametrize the symmetry group of the $n$-dimensional
Berwald-Mo\'{o}r metric. Some properties of this Lie group are studied, and
its corresponding Lie algebra is computed.
\end{abstract}

\textbf{Mathematics Subject Classification (2010):} 53C60, 54H15, 17B45.

\textbf{Key words and phrases:} tangent bundle, $n$-dimensional Berwald-Mo%
\'{o}r metric, symmetry group, Lie group, Lie algebra.

\section{Introduction}

The geometrical Berwald-Mo\'{o}r structure (\cite{Berwald}, \cite{Moor-1})
was intensively investigated by P.K. Rashevski (\cite{Rashevski}) and
further physically fundamented and developed by G.S. Asanov (\cite{Asanov[1]}%
), D.G. Pavlov and G.I. Garas'ko (\cite{Pavlov(2)[10]}, \cite{Garasko-2}).
These physical studies emphasize the importance of the Finsler geometry
characterized by the total equality in rights of all non-isotropic
directions in the theory of space-time structure, gravitation and
electromagnetism. In such a context, one underlines the important role
played by the Berwald-Mo\'{o}r metric%
\begin{equation*}
F_{n}:TM^{n}\rightarrow \mathbb{R},\mathbb{\qquad }F_{n}(y)=\sqrt[n]{%
y^{1}y^{2}...y^{n}},\qquad n\geq 2,
\end{equation*}%
whose Finslerian geometry is studied on tangent bundles by M. Matsumoto and
H. Shimada (\cite{Mats-Shimada}), and, in a jet geometrical approach, by V.
Balan and M. Neagu (\cite{Balan-Neagu}). It is a well-known fact that, from
a physical point of view, an Einstein relativistic law says that the form of
all physical laws must be the same in any inertial reference frame (local
chart of coordinates). For such a reason, we study in this paper the
geometrical coordinate transformations which keep unchanged the Berwald-Mo%
\'{o}r metric of order $n\geq 3.$ The particular two and three dimensional
cases are deeply studied in \cite{Neagu-Raeisi}. Notice that the geometrical
translation of the previous Einstein's physical law is that any geometrical
object used in that physical theory must have the same local form in any
local chart of coordinates.

\section{The symmetry group of the Berwald-Mo\'{o}r metric of order $n$}

We remind that $(x,y)=(x^{i},y^{i})$ are the coordinates of the tangent
bundle $TM^{n}$ (associated to an $n$-dimensional real manifold $M^{n}$),
which transform by the rules (the Einstein convention of summation is
assumed everywhere):%
\begin{equation}
\widetilde{x}^{i}=\widetilde{x}^{i}(x^{j}),\qquad \widetilde{y}^{i}=\dfrac{%
\partial {\widetilde{x}^{i}}}{\partial {x^{j}}}y^{j},  \label{tr-rules}
\end{equation}%
where $i,j=\overline{1,n}$ and rank $(\partial \widetilde{x}^{i}/\partial
x^{j})=n$. The transformation rules (\ref{tr-rules}) rewrite explicitly as

\begin{equation}
\begin{array}{c}
\widetilde{{x}}^{1}=\widetilde{{x}}^{1}(x^{1},...,x^{n}),\hspace{1mm}%
\widetilde{{x}}^{2}=\widetilde{{x}}^{2}(x^{1},...,x^{n}),...,\hspace{1mm}%
\widetilde{{x}}^{n}=\widetilde{{x}}^{n}(x^{1},...,x^{n}),\medskip  \\ 
\begin{array}{l}
\widetilde{y}^{1}=\partial _{1}^{1}y^{1}+\partial _{2}^{1}y^{2}+\partial
_{3}^{1}y^{3}+....+\partial _{n}^{1}y^{n}\medskip  \\ 
\widetilde{y}^{2}=\partial _{1}^{2}y^{1}+\partial _{2}^{2}y^{2}+\partial
_{3}^{2}y^{3}+....+\partial _{n}^{2}y^{n} \\ 
\vdots  \\ 
\widetilde{y}^{n}=\partial _{1}^{n}y^{1}+\partial _{2}^{n}y^{2}+\partial
_{3}^{n}y^{3}+....+\partial _{n}^{n}y^{n},%
\end{array}%
\end{array}
\label{tr-coordinates-order n}
\end{equation}%
\newline
where we used the notations $\partial _{j}^{i}:=\partial \widetilde{x}%
^{i}/\partial x^{j},$ and we have $\det \left( \partial _{j}^{i}\right) \neq
0$.

Let us consider now the Berwald-Mo\'{o}r metric of order $n$ on the tangent
bundle $TM^{n}$, which is expressed by 
\begin{equation}
F_{n}(y)=\sqrt[n]{y^{1}y^{2}y^{3}...y^{n}}.  \label{n-Berwal-Moor}
\end{equation}

To have a global geometrical character of the Berwald-Mo\'{o}r metric (\ref%
{n-Berwal-Moor}), we must have $F_{n}(\widetilde{y})=F_{n}(y).$ It means
that $\widetilde{y}^{1}\widetilde{y}^{2}...\widetilde{y}%
^{n}=y^{1}y^{2}...y^{n}$.

\begin{theorem}
\label{n dimension}For $n\geq 3$, a transformation of coordinates (\ref%
{tr-coordinates-order n}) invariates the Berwald-Mo\'{o}r metric (\ref%
{n-Berwal-Moor}) if and only if there exist some arbitrary real numbers $%
a_{1},$ $a_{2},...,$ $a_{n}$ verifying the equality $\Pi _{i=1}^{n}a_{i}=1$,
together with a permutation $\sigma \in S_{n}$ of the set $\{1,2,...,n\}$,
such that 
\begin{equation}
\widetilde{\mathcal{X}}=\mathcal{P}_{\sigma }\left(
a_{1},a_{2},...,a_{n}\right) \cdot \mathcal{X}+\mathcal{X}_{0},  \label{16}
\end{equation}%
where 
\begin{equation*}
\mathcal{X}_{0}=\left( 
\begin{array}{c}
x_{0}^{1} \\ 
x_{0}^{2} \\ 
\vdots  \\ 
x_{0}^{n}%
\end{array}%
\right) \in M_{n,1}(\mathbb{R}),\qquad \mathcal{X}=\left( 
\begin{array}{c}
x^{1} \\ 
x^{2} \\ 
\vdots  \\ 
x^{n}%
\end{array}%
\right) ,\qquad \widetilde{\mathcal{X}}=\left( 
\begin{array}{c}
\widetilde{{x}}^{1} \\ 
\widetilde{{x}}^{2} \\ 
\vdots  \\ 
\widetilde{{x}}^{n}%
\end{array}%
\right) ,
\end{equation*}%
and the matrix $\mathcal{P}_{\sigma }\left( a_{1},a_{2},...,a_{n}\right) $
has all entries equal to zero except the entries%
\begin{equation*}
p_{1\sigma (1)}=a_{1},\text{ }p_{2\sigma (2)}=a_{2},\text{ }...\text{ },%
\text{ }p_{n\sigma (n)}=a_{n}.
\end{equation*}
\end{theorem}

\begin{proof}
The transformation of coordinates (\ref{tr-coordinates-order n}) invariates
the Berwald-Mo\'{o}r metric (\ref{n-Berwal-Moor}) if and only if 
\begin{equation}
\left\{ 
\begin{array}{l}
\sum_{\tau \in S_{n}}\partial _{\tau (1)}^{1}\partial _{\tau
(2)}^{2}...\partial _{\tau (n)}^{n}=1\medskip \\ 
\partial _{k_{1}}^{1}\partial _{k_{2}}^{2}...\partial _{k_{n}}^{n}=0,%
\end{array}%
\right.  \label{system-n}
\end{equation}%
for any $k_{1},k_{2},...,k_{n}\in \{1,2,...,n\}$ such that $%
\{k_{1},k_{2},...,k_{n}\}\neq \{1,2,...,n\}$ (this means that at least two
indices $k_{r}$ and $k_{s}$ are equal). Then, there exists a permutation $%
\sigma $ of the set $\{1,2,...,n\}$ such that%
\begin{equation*}
\partial _{\sigma (i)}^{i}\neq 0,\quad \forall \text{ }i=\overline{1,n}.
\end{equation*}%
This is because if we suppose that there exists an index $i_{0}\in
\{1,2,...,n\}$ such that $\partial _{\sigma (i_{0})}^{i_{0}}=0$ for any
permutation $\sigma \in S_{n}$, then (using the first equation of the system
(\ref{system-n})) we deduce that we have $0=1.$ Contradiction!

Let us prove now that for any $i\in \{1,2,...,n\}$ we have%
\begin{equation*}
\partial _{k}^{i}=0,\quad \forall \text{ }k\in \{1,2,...,n\}\backslash
\{\sigma (i)\}.
\end{equation*}%
On the contrary, suppose there exist $i,k\in \{1,2,...,n\}$, with $k\neq {%
\sigma (i)}$, such that $\partial _{k}^{i}\neq 0$. Because we obviously have 
$k=\sigma (j),$ where $j\neq i$, it follows that there exist two different
indices $i,j\in \{1,2,...,n\}$ such that $\partial _{\sigma (j)}^{i}\neq 0.$
Because we have the inequality $n\geq 3$, it follows that there exists an
arbitrary third index $q\in \{1,2,...,n\}$ which is different by $i$ and $j$%
. Consequently, applying the second equation from the system (\ref{system-n}%
) for an arbitrary index $k_{q}\in \{1,2,...,n\}$ and for%
\begin{equation*}
k_{i}=k_{j}=\sigma (j),\quad k_{p}=\sigma (p),\text{ }\forall \text{ }p\in
\{1,2,...,n\}\backslash \{i,j,q\},
\end{equation*}%
we find%
\begin{equation*}
\underset{\overset{\nparallel }{0}}{\underbrace{\partial _{\sigma
(1)}^{1}...\partial _{\sigma (i-1)}^{i-1}\partial _{\sigma (j)}^{i}\partial
_{\sigma (i+1)}^{i+1}...\partial _{\sigma (j-1)}^{j-1}\partial _{\sigma
(j)}^{j}\partial _{\sigma (j+1)}^{j+1}...\partial _{\sigma (q-1)}^{q-1}}}%
\partial _{k_{q}}^{q}\underset{\overset{\nparallel }{0}}{\underbrace{%
\partial _{\sigma (q+1)}^{q+1}...\partial _{\sigma (n)}^{n}}}=0.
\end{equation*}%
It follows that we have%
\begin{equation*}
\partial _{k_{q}}^{q}=0,\quad\forall \text{ }k_{q}=\overline{1,n}.
\end{equation*}%
This implies that $\det \left( \partial _{j}^{i}\right) =0$. Contradiction!

In conclusion, we deduce that a transformation of coordinates (\ref%
{tr-coordinates-order n}) invariates the Berwald-Mo\'{o}r metric (\ref%
{n-Berwal-Moor}) if and only if it satisfies the following conditions:%
\begin{equation}
\partial _{\sigma (1)}^{1}\partial _{\sigma (2)}^{2}...\partial _{\sigma
(n)}^{n}=1,  \label{eq-1}
\end{equation}%
\begin{equation}
\partial _{1}^{i}=\partial _{2}^{i}=...=\partial _{\sigma
(i)-1}^{i}=\partial _{\sigma (i)+1}^{i}=...=\partial _{n-1}^{i}=\partial
_{n}^{i}=0,\quad \forall \text{ }i=\overline{1,n}.  \label{eq-2}
\end{equation}%
The equations (\ref{eq-1}) and (\ref{eq-2}) imply that $\tilde{x}^{i}=\tilde{%
x}^{i}(x^{\sigma (i)})$ and%
\begin{equation*}
\partial _{\sigma (i)}^{i}=p_{i\sigma (i)}:=a_{i}\in \mathbb{R}\backslash
\{0\}.
\end{equation*}%
In other words, we get the affine transformations (no sum by $i$):%
\begin{equation*}
\tilde{x}^{i}=p_{i\sigma (i)}x^{\sigma (i)}+x_{0}^{i},\qquad \forall \text{ }%
i=\overline{1,n},
\end{equation*}%
where $x_{0}^{i}\in \mathbb{R},$ $\forall $ $i=\overline{1,n}$.
\end{proof}

\begin{corollary}
The set of the local transformations of coordinates that invariates the
Berwald-Mo\'{o}r metric (\ref{n-Berwal-Moor}) has an algebraic structure of
a non-abelian group with respect to the operation of composition of
functions.
\end{corollary}

\begin{proof}
Firstly, it is important to notice that the following matrix properties are
true:%
\begin{equation*}
\mathcal{P}_{\sigma }\left( a_{1},a_{2},...,a_{n}\right) =\left( a_{i}\delta
_{j\sigma (i)}\right) _{i,j=\overline{1,n}},
\end{equation*}%
\begin{equation*}
\mathcal{P}_{\sigma }\left( a_{1},a_{2},...,a_{n}\right) =\mathcal{P}_{\tau
}\left( b_{1},b_{2},...,b_{n}\right) \Leftrightarrow \sigma =\tau \text{ and 
}a_{i}=b_{i}\text{ }\forall \text{ }i=\overline{1,n},
\end{equation*}%
\begin{equation*}
\mathcal{P}_{\sigma }\left( a_{1},a_{2},...,a_{n}\right) \cdot \mathcal{P}%
_{\tau }\left( b_{1},b_{2},...,b_{n}\right) =\mathcal{P}_{\tau \circ \sigma
}\left( a_{1}b_{\sigma (1)},a_{2}b_{\sigma (2)},...,a_{n}b_{\sigma
(n)}\right) ,
\end{equation*}%
\begin{equation*}
\mathcal{P}_{e}\left( 1,1,...,1\right) =I_{n},\qquad \det \left[ \mathcal{P}%
_{\sigma }\left( a_{1},a_{2},...,a_{n}\right) \right] =\varepsilon (\sigma ),
\end{equation*}%
\begin{equation*}
\left[ \mathcal{P}_{\sigma }\left( a_{1},a_{2},...,a_{n}\right) \right]
^{-1}=\mathcal{P}_{\sigma ^{-1}}\left( a_{\sigma ^{-1}(1)}^{-1},a_{\sigma
^{-1}(2)}^{-1},...,a_{\sigma ^{-1}(n)}^{-1}\right) ,
\end{equation*}%
where $e$ is the identity permutation, $I_{n}$ is the identity matrix, and $%
\varepsilon (\sigma )$ is the signature of the permutation $\sigma $.

Let $\mathfrak{T}$ be the set of the local transformations of coordinates
that invariates the Berwald-Mo\'{o}r metric (\ref{n-Berwal-Moor}). Let also $%
\mathcal{S}$ and $\mathcal{T}$ be two arbitrary transformations from $%
\mathfrak{T}$. Then we have%
\begin{equation*}
\begin{array}{l}
\mathcal{S}:\mathcal{\tilde{X}}=\mathcal{P}_{\sigma }\left(
a_{1},a_{2},...,a_{n}\right) \cdot \mathcal{X}+\mathcal{X}_{0},\medskip  \\ 
\mathcal{T}:\mathcal{X}=\mathcal{P}_{\tau }\left(
b_{1},b_{2},...,b_{n}\right) \cdot \overline{\mathcal{X}}+\overline{\mathcal{%
X}}_{0}%
\end{array}%
\end{equation*}%
and%
\begin{equation*}
\mathcal{S}\circ \mathcal{T}:\widetilde{\mathcal{X}}=\mathcal{P}_{\sigma
}\left( a_{1},a_{2},...,a_{n}\right) \cdot \left[ \mathcal{P}_{\tau }\left(
b_{1},b_{2},...,b_{n}\right) \cdot \overline{\mathcal{X}}+\overline{\mathcal{%
X}}_{0}\right] +\mathcal{X}_{0}=
\end{equation*}%
\begin{equation*}
=\left[ \mathcal{P}_{\sigma }\left( a_{1},a_{2},...,a_{n}\right) \cdot 
\mathcal{P}_{\tau }\left( b_{1},b_{2},...,b_{n}\right) \right] \cdot 
\overline{\mathcal{X}}+\overline{\mathcal{X}}_{1}=
\end{equation*}%
\begin{equation*}
=\mathcal{P}_{\tau \circ \sigma }\left( a_{1}b_{\sigma (1)},a_{2}b_{\sigma
(2)},...,a_{n}b_{\sigma (n)}\right) \cdot \overline{\mathcal{X}}+\overline{%
\mathcal{X}}_{1},
\end{equation*}%
where%
\begin{equation*}
\overline{\mathcal{X}}_{1}=\mathcal{P}_{\sigma }\left(
a_{1},a_{2},...,a_{n}\right) \cdot \overline{\mathcal{X}}_{0}+\mathcal{X}%
_{0}.
\end{equation*}%
Moreover, the neutral transformation element is%
\begin{equation*}
\mathcal{E}:\mathcal{\tilde{X}}=\mathcal{P}_{e}\left( 1,1,...,1\right) \cdot 
\mathcal{X},
\end{equation*}%
and we have%
\begin{equation*}
\begin{array}{l}
\mathcal{S}^{-1}:\mathcal{X}=\left[ \mathcal{P}_{\sigma }\left(
a_{1},a_{2},...,a_{n}\right) \right] ^{-1}\cdot \left[ \widetilde{\mathcal{X}%
}-\mathcal{X}_{0}\right] =\medskip  \\ 
=\mathcal{P}_{\sigma ^{-1}}\left( a_{\sigma ^{-1}(1)}^{-1},a_{\sigma
^{-1}(2)}^{-1},...,a_{\sigma ^{-1}(n)}^{-1}\right) \cdot \widetilde{\mathcal{%
X}}+\widetilde{\mathcal{X}}_{1},%
\end{array}%
\end{equation*}%
where%
\begin{equation*}
\widetilde{\mathcal{X}}_{1}=-\mathcal{P}_{\sigma ^{-1}}\left( a_{\sigma
^{-1}(1)}^{-1},a_{\sigma ^{-1}(2)}^{-1},...,a_{\sigma ^{-1}(n)}^{-1}\right)
\cdot \mathcal{X}_{0}.
\end{equation*}

In conclusion, the set $\mathfrak{T}$ is a group with respect to the
operation of composition of functions.

Let $\sigma $ and $\tau $ be two permutations such that $\tau \circ \sigma
\neq \sigma \circ \tau $. If we take, for instance, $\mathcal{S}_{0}$ and $%
\mathcal{T}_{0}$ two transformations from $\mathfrak{T}$ having the
homogenous form%
\begin{equation*}
\mathcal{S}_{0}(\mathcal{X)}=\mathcal{P}_{\sigma }\left(
a_{1},a_{2},...,a_{n}\right) \cdot \mathcal{X},\qquad \mathcal{T}_{0}(%
\mathcal{X)}=\mathcal{P}_{\tau }\left( b_{1},b_{2},...,b_{n}\right) \cdot 
\mathcal{X},
\end{equation*}%
then we get 
\begin{equation*}
\left( \mathcal{S}_{0}\circ \mathcal{T}_{0}\right) (\mathcal{X)}=\mathcal{P}%
_{\sigma }\left( a_{1},a_{2},...,a_{n}\right) \cdot \left[ \mathcal{P}_{\tau
}\left( b_{1},b_{2},...,b_{n}\right) \cdot \mathcal{X}\right] =
\end{equation*}%
\begin{equation*}
=\left[ \mathcal{P}_{\sigma }\left( a_{1},a_{2},...,a_{n}\right) \cdot 
\mathcal{P}_{\tau }\left( b_{1},b_{2},...,b_{n}\right) \right] \cdot 
\mathcal{X=}
\end{equation*}%
\begin{equation*}
=\mathcal{P}_{\tau \circ \sigma }\left( a_{1}b_{\sigma (1)},a_{2}b_{\sigma
(2)},...,a_{n}b_{\sigma (n)}\right) \cdot \mathcal{X\neq }
\end{equation*}%
\begin{equation*}
\neq \mathcal{P}_{\sigma \circ \tau }\left( b_{1}a_{\tau (1)},b_{2}a_{\tau
(2)},...,b_{n}a_{\tau (n)}\right) \cdot \mathcal{X=}
\end{equation*}%
\begin{equation*}
=\left[ \mathcal{P}_{\tau }\left( b_{1},b_{2},...,b_{n}\right) \cdot 
\mathcal{P}_{\sigma }\left( a_{1},a_{2},...,a_{n}\right) \right] \cdot 
\mathcal{X=}\left( \mathcal{T}_{0}\circ \mathcal{S}_{0}\right) (\mathcal{X)}.
\end{equation*}%
Therefore, the group $\left( \mathfrak{T},\circ \right) $ is non-abelian.
\end{proof}

For any permutation $\sigma \in S_{n}$, let us denote by $\mathcal{W}%
_{\sigma }$ the set of all matrices of type $\mathcal{P}_{\sigma }\left(
a_{1},a_{2},...,a_{n}\right) $. We recall that such a matrix has all entries
equal to zero except the entries $a_{1\sigma (1)}=a_{1},$ $a_{2\sigma
(2)}=a_{2},$ $...,$ $a_{n\sigma (n)}=a_{n}$, verifying the equality $\Pi
_{i=1}^{n}a_{i}=1$. Let us consider the matrix%
\begin{equation*}
E_{\sigma }:=\mathcal{P}_{\sigma }\left( 1,1,...,1\right) \in \mathcal{W}%
_{\sigma }.
\end{equation*}%
In such a context, we can prove the following important algebraic result of
characterization:

\begin{proposition}
Let $\sigma \in S_{n}$ be an arbitrary permutation of the set $\{1,2,...,n\}$%
. Then, an arbitrary matrix $\mathfrak{X}\in M_{n}(\mathbb{R})$ belongs to
the set $\mathcal{W}_{\sigma ^{-1}}$ if and only if the following statements
are true:

\begin{enumerate}
\item[\emph{(1)}] $\det \left( \mathfrak{X\cdot }E_{\sigma }\right) =1;$

\item[\emph{(2)}] The vectors%
\begin{equation*}
{\mathbf{e_{1}}}=(1,0,0,...,0),\;{\mathbf{e_{2}}}=(0,1,0,...,0),\text{ }%
...,\;{\mathbf{e_{n}}}=(0,0,0,...,1)
\end{equation*}%
are eigenvectors of the matrix $\left( \mathfrak{X\cdot }E_{\sigma }\right) $%
.
\end{enumerate}
\end{proposition}

\begin{proof}
An arbitrary matrix $A\in M_{n}(\mathbb{R})$ is a diagonal matrix of the form%
\begin{equation*}
A=\left( 
\begin{array}{ccccc}
\lambda _{1} & 0 & 0 & \cdots  & 0 \\ 
0 & \lambda _{2} & 0 & \cdots  & 0 \\ 
0 & 0 & \ddots  & \cdots  & 0 \\ 
\vdots  & \vdots  & \vdots  & \ddots  & \vdots  \\ 
0 & 0 & 0 & \cdots  & \lambda _{n}%
\end{array}%
\right) ,
\end{equation*}%
with $\lambda _{1}\lambda _{2}...\lambda _{n}=1$, if and only if $\det A=1$
and it has the vectors ${\mathbf{e_{1}}},\;{\mathbf{e_{2}}},$ $...,\;{%
\mathbf{e_{n}}}$ as eigenvectors. Using now the properties (1) and (2), we
deduce that%
\begin{equation*}
\mathfrak{X\cdot }E_{\sigma }=A\Leftrightarrow \mathfrak{X\cdot }\mathcal{P}%
_{\sigma }\left( 1,1,...,1\right) =\mathcal{P}_{e}(\lambda _{1},...,\lambda
_{n})\Leftrightarrow 
\end{equation*}%
\begin{equation*}
\Leftrightarrow \mathfrak{X}=\mathcal{P}_{e}(\lambda _{1},...,\lambda
_{n})\cdot \left[ \mathcal{P}_{\sigma }\left( 1,1,...,1\right) \right] ^{-1}=
\end{equation*}%
\begin{equation*}
=\mathcal{P}_{e}(\lambda _{1},...,\lambda _{n})\cdot \mathcal{P}_{\sigma
^{-1}}\left( 1,1,...,1\right) =\mathcal{P}_{\sigma ^{-1}}\left( \lambda
_{1},...,\lambda _{n}\right) \in \mathcal{W}_{\sigma ^{-1}}.
\end{equation*}
\end{proof}

In what follows we prove that the set of matrices $D_{n}^{1}(\mathbb{R)}:%
\mathcal{=W}_{e}$, where $e$ is the identity permutation, is a Lie group of
dimension $\left( n-1\right) $ with respect to the multiplication of
matrices.

\begin{proposition}
The set of matrices $D_{n}^{1}(\mathbb{R)}$ has a structure of commutative
Lie group with respect to the multiplication of matrices. The dimension as
manifold of the Lie group $D_{n}^{1}(\mathbb{R)}$ is $\left( n-1\right) ,$
and the corresponding Lie algebra is%
\begin{equation}
L(D_{n}^{1}(\mathbb{R)})=\{A=\text{\emph{diag}}(a_{1},...,a_{n})\text{ }|%
\text{ \emph{Trace}}(A)=0\}:=d_{n}^{1}(\mathbb{R)}.  \label{Lie algebra}
\end{equation}
\end{proposition}

\begin{proof}
It is obvious that%
\begin{equation*}
D_{n}^{1}(\mathbb{R)}=\mathcal{W}_{e}=\left\{ \mathcal{P}_{e}\left(
a_{1},a_{2},...,a_{n}\right) \text{ }|\text{ }a_{i}\in \mathbb{R}\right\} =
\end{equation*}%
\begin{equation*}
=\left\{ A=\text{diag}(a_{1},...,a_{n})\text{ }|\text{ }\det A=a_{1}\cdot
...\cdot a_{n}=1\right\}
\end{equation*}%
is a commutative subgroup of the special linear group $SL_{n}(\mathbb{R})$.
It is also easy to see that we have 
\begin{equation*}
D_{n}^{1}(\mathbb{R)}=\left\{ \left. A=\text{diag}\left( a_{1},...,a_{n-1},%
\dfrac{1}{a_{1}\cdot ...\cdot a_{n-1}}\right) \text{ }\right\vert \text{ }%
a_{1},\text{ }...,\text{ }a_{n-1}\in \mathbb{R}\backslash \{0\}\right\} .
\end{equation*}

Let $\phi _{U}:D_{n}^{1}(\mathbb{R)}\rightarrow U\subset M:=\mathbb{R}%
^{n-1}\backslash \left\{ \left( a_{1},...,a_{n-1}\right) \text{ }|\text{ }%
a_{1}\cdot ...\cdot a_{n-1}=0\right\} $ be the bijection defined by $\phi
_{U}\left( A\right) =\left( a_{1},...,a_{n-1}\right) $, where $U$ is an
arbitrary local chart on $M$. It folows that we can endow $D_{n}^{1}(\mathbb{%
R)}$ with a differentiable structure of dimension $\left( n-1\right) $ such
that all maps $\phi _{U}$ to become diffeomorfisms. Therefore, the mapping $%
\mu :D_{n}^{1}(\mathbb{R)}\times D_{n}^{1}(\mathbb{R)}\rightarrow D_{n}^{1}(%
\mathbb{R)}$, defined by%
\begin{equation*}
\mu (A,B)=A^{-1}\cdot B,
\end{equation*}%
can be locally rewritten as%
\begin{equation*}
\widetilde{\mu }:\mathbb{R}^{2n-2}\backslash \left\{ \left(
a_{1},...,a_{n-1},b_{1},...,b_{n-1}\right) \text{ }|\text{ }a_{1}\cdot
...\cdot a_{n-1}\cdot b_{1}\cdot ...\cdot b_{n-1}=0\right\} \rightarrow 
\mathbb{R}^{n-1},
\end{equation*}%
where%
\begin{equation*}
\widetilde{\mu }\left( a_{1},...,a_{n-1},b_{1},...,b_{n-1}\right) =\left( 
\frac{b_{1}}{a_{1}},\frac{b_{2}}{a_{2}},...,\frac{b_{n-1}}{a_{n-1}}\right) .
\end{equation*}%
It is obvious that $\widetilde{\mu }$ is a smooth map on the open domain%
\begin{equation*}
D=\mathbb{R}^{2n-2}\backslash \left\{ \left(
a_{1},...,a_{n-1},b_{1},...,b_{n-1}\right) \text{ }|\text{ }a_{1}\cdot
...\cdot a_{n-1}\cdot b_{1}\cdot ...\cdot b_{n-1}=0\right\} .
\end{equation*}

In conclusion, $D_{n}^{1}(\mathbb{R)}$ is a commutative Lie group of
dimension $\left( n-1\right) $.

In order to compute the Lie algebra $L(D_{n}^{1}(\mathbb{R)})=d_{n}^{1}(%
\mathbb{R)}$, let us consider an arbitrary curve%
\begin{equation*}
\alpha :(-\epsilon ,\epsilon )\rightarrow D_{n}^{1}(\mathbb{R)},\qquad
\alpha (t)=A(t)=\text{diag}(a_{1}(t),...,a_{n}(t))\in D_{n}^{1}(\mathbb{R)},
\end{equation*}%
where $\alpha (0)=A(0)=I_{n}$. It follows that we have%
\begin{equation*}
\dot{\alpha}(0)=\left. \frac{dA}{dt}\right\vert _{t=0}=\text{diag}(\dot{a}%
_{1}(0),...,\dot{a}_{n}(0)).
\end{equation*}%
Because we have $\det A(t)=1$, we deduce that%
\begin{equation*}
0=\left. \frac{d\left[ \det A\right] }{dt}\right\vert _{t=0}=\det \left[ 
\text{diag}(\dot{a}_{1}(0),1,...,1)\right] +
\end{equation*}%
\begin{equation*}
+\det \left[ \text{diag}(1,\dot{a}_{2}(0),1,...,1)\right] +\text{ }...\text{ 
}+\det \left[ \text{diag}(1,1,...,1,\dot{a}_{n}(0))\right] =
\end{equation*}%
\begin{equation*}
=\dot{a}_{1}(0)+\dot{a}_{1}(0)+...+\dot{a}_{n}(0).
\end{equation*}%
This means that the Lie algebra of $D_{n}^{1}(\mathbb{R)}$ is given by (\ref%
{Lie algebra}).
\end{proof}

\begin{remark}
A basis of the Lie algebra $d_{n}^{1}(\mathbb{R)}$ is given by $\left(
E_{i}\right) _{i=\overline{1,n-1}},$ where 
\begin{equation*}
E_{i}=\left( 
\begin{array}{ccccccccc}
0 & 0 & \cdots  & \cdots  & \cdots  & \cdots  & \cdots  & 0 & 0 \\ 
0 & 0 & \cdots  & \cdots  & \cdots  & \cdots  & \cdots  & 0 & 0 \\ 
\vdots  & \vdots  & \ddots  & \ddots  & \ddots  & \ddots  & \ddots  & \vdots 
& \vdots  \\ 
0 & 0 & \ddots  & 0 & \ddots  & \ddots  & \ddots  & 0 & 0 \\ 
0 & 0 & \ddots  & \ddots  & 1 & \ddots  & \ddots  & 0 & 0 \\ 
0 & 0 & \ddots  & \ddots  & \ddots  & 0 & \ddots  & 0 & 0 \\ 
\vdots  & \vdots  & \ddots  & \ddots  & \ddots  & \ddots  & \ddots  & \vdots 
& \vdots  \\ 
0 & 0 & \cdots  & \cdots  & \cdots  & \cdots  & \cdots  & 0 & 0 \\ 
0 & 0 & \cdots  & \cdots  & \cdots  & \cdots  & \cdots  & 0 & -1%
\end{array}%
\right) .
\end{equation*}%
Notice that in the matrix $E_{i}$ the number $1$ appears on the position $%
(i,i)$. In other words, we can briefly write%
\begin{equation*}
E_{i}=\left( \delta _{ri}\delta _{si}-\delta _{rn}\delta _{sn}\right) _{r,s=%
\overline{1,n}},\quad \forall \text{ }i=\overline{1,n-1}.
\end{equation*}%
Consequently, it is easy to see now that \textbf{all structure constants of
the Lie algebra }$d_{n}^{1}(R)$\textbf{\ are equal to zero}.
\end{remark}

\noindent\textbf{Acknowledgements.} The authors of this paper thank
Professor M. P\u{a}un for the useful geometrical discussions on this
research topic.

\end{document}